\documentclass{article}
\usepackage{amsfonts}
\usepackage{amsmath}
\usepackage{amssymb}
\usepackage{graphicx}
\usepackage{cite}

\newtheorem{theorem}{Theorem}

\newtheorem{corollary}[theorem]{Corollary}

\newtheorem{lemma}[theorem]{Lemma}

\newenvironment{proof}[1][Proof]{\noindent\textbf{#1.} }{\ \rule{0.5em}{0.5em}}
\input{tcilatex}

\begin{document}

\title{Conjugate operators to operators of stochastic integration}
\author{I.P. Smirnov \\
Institute of Applied Physics RAS\\
46 Ul'anova Street, Nizhny Novgorod, Russia}
\date{}
\maketitle

\begin{abstract}
The conjugate problem in stochastic optimal control can be formulated in
terms of operators conjugated to the operators of stochastic integration 
\cite{V801, Bismut,V802}. In this paper we study some of such operators
acting on the spaces of progressively measurable random functions.
\end{abstract}


\section{Introduction}

\subsection{Functional spaces}

Let $\left( \Omega ,\mathfrak{F},\mathbf{P}\right) $ be a complete
probability space, $\left\{ \mathfrak{F}_{t},\;t\in \left[ 0,1\right]
\right\} $ the right continuous stream of subfields of the field $\mathfrak{F%
}$. We use the following notations: $\mathfrak{B}^{n}$ is the $\sigma $%
-field of Borel sets on $R^{n}$, $\limfunc{mes}$ is the Lebesgue measure on $%
\mathfrak{B}^{n}$, $\mathfrak{B}^{1}\times \mathfrak{F}$ and $\limfunc{mes}%
\times $ $\mathbf{P}$ are the direct products of the fields $\mathfrak{B}^{1}
$, $\mathfrak{F}$ and measures $\limfunc{mes}$,$\mathbf{P}$, respectively, $%
\mathfrak{BF\subset B}^{1}\times \mathfrak{F}$ is the $\sigma $-field of
progressively measurable with respect to $\left\{ \mathfrak{F}_{t}\right\} $
subsets of the product $\left[ 0,1\right] \times \Omega $.

Let us introduce the main functional spaces used below. For $p\in\left[
1,\infty\right) $ let denote by

\begin{enumerate}
\item $L_{p}\left( \Omega\right) $ the Banach space of $\mathfrak{F}$%
-measured random values $\xi\left( \omega\right) $, $\omega\in\Omega$ ( i.e.
classes of values equivalent to the each other with respect to probability
measure $\mathbf{P}$) with the norm%
\begin{equation*}
\left\Vert \xi\right\Vert _{L_{p}\left( \Omega\right) }=\left\{ \mathbf{M}%
\left\vert \xi\right\vert ^{p}\right\} ^{1/p}
\end{equation*}
($\mathbf{M}$ is the expectation under the measure $\mathbf{P}$\textbf{)};

\item $E_{p}$ the set of $\mathfrak{BF}$-measurable random functions $%
f\left( t,\omega\right) $,$\,t\in\left[ 0,1\right] $, $\omega\in\Omega$ such
that $\mathbf{P}$-a.s. 
\begin{equation*}
\int\limits_{0}^{1}\left\vert f\left( t,\omega\right) \right\vert
^{p}\,dt<\infty; 
\end{equation*}

\item $L_{p}$ the Banach space of $\mathfrak{BF}$-measurable functions $%
f\left( t,\omega\right) $ (classes of functions equivalent to the each other
with respect to measure $\limfunc{mes}\times\mathbf{P}$) with the norm%
\begin{equation*}
\left\Vert f\right\Vert _{p}=\left\{ \mathbf{M}\int\limits_{0}^{1}\left\vert
f\left( t,\omega\right) \right\vert ^{p}\,dt\right\} ^{1/p}; 
\end{equation*}

\item $N_{p}$ the Banach space of $\mathfrak{BF}$-measurable functions $%
f\left( t,\omega\right) $ (classes of functions) with the norm 
\begin{equation*}
\left\Vert f\right\Vert _{N_{p}}=\left\{ \mathbf{M}\left( \int
\limits_{0}^{1}\left\vert f\left( t,\omega\right) \right\vert
^{2}\,dt\right) ^{p/2}\right\} ^{1/p}; 
\end{equation*}

\item $H_{p}$ the Banach space of $\mathfrak{BF}$-measurable functions $%
f\left( t,\omega\right) $ (classes of functions) with the norm 
\begin{equation*}
\left\vert f\right\vert _{p}=\left\{ \sup_{0\leqslant t\leqslant1}\mathbf{M}%
\left\vert f\left( t,\omega\right) \right\vert ^{p}\right\} ^{1/p}; 
\end{equation*}

\item $DH_{p}$ the normed space of $\mathfrak{BF}$-measurable functions $%
f\left( t,\omega\right) $ (classes of $\mathbf{P}$- indistinguishable
functions) having $\mathbf{P}$-a.s. trajectories of the class $D$ \footnote{%
A function $\varphi:\left[ 0,1\right] \rightarrow R^{1}$ belongs to the
class $D$ if it is right continious and has finite lower limits everywhere
on the segment $\left[ 0,1\right] $.} with the norm 
\begin{equation*}
\left\vert f\right\vert _{p}^{\left( d\right) }=\left\{ \mathbf{M}%
\sup_{0\leqslant t\leqslant1}\left\vert f\left( t,\omega\right) \right\vert
^{p}\right\} ^{1/p}; 
\end{equation*}

\item $CH_{p}$ the subset of $DH_{p}$ of all continuous by $t$ a.s.
functions $f\left( t,\omega\right) $;

\item $\Pi\left( \cdot\right) $ the $\sigma$-finite measure on $\mathfrak{B}%
^{n}$,

\begin{enumerate}
\item $E_{p}\left( \Pi\right) $ the set of all $\mathfrak{BF}\times 
\mathfrak{B}^{n}$-measurable functions $\psi\left( t,\omega,y\right) $,$%
\;t\in\left[ 0,1\right] $, $\omega\in\Omega$, $y\in R^{l}$ such that%
\begin{equation*}
\int\limits_{0}^{1}\int\left\vert \psi\left( t,\omega,y\right) \right\vert
^{p}\Pi\left( dy\right) <\infty\ a.s. 
\end{equation*}

\item $L_{p}\left( \Pi\right) $ the Banach space $\mathfrak{BF}\times%
\mathfrak{B}^{n}$-measurable functions $\psi\left( t,\omega,y\right) $
(classes of functions) with the norm%
\begin{equation*}
\left\Vert \psi\right\Vert _{L_{p}\left( \Pi\right) }=\left\{ \mathbf{M}%
\int\limits_{0}^{1}\left[ \int\left\vert \psi\right\vert ^{p}\Pi\left(
dy\right) +\left( \int\left\vert \psi\right\vert ^{2}\Pi\left( dy\right)
\right) ^{p/2}\right] \right\} ^{1/p}. 
\end{equation*}
\end{enumerate}
\end{enumerate}

We shall consider the embeddings of the type $DH_{p}\subset H_{p}\subset
L_{p}$. In this case an element of a slender space is identified with the
class of equivalent elements from the wider space. Conversely, if we state
that an element of a wide space belongs to a slender space it means that the
corresponding class of equivalent function includes an element from the
slender space. Note that $DH_{p}$ and $H_{p}$ are everywhere dense in $L_{p}$%
.

Further, $L_{p}\left( \Omega\right) $, $L_{p}$, $N_{p}$, $L_{p}\left(
\Pi\right) $ are the Banach ideal spaces of measurable functions with
order-type continuous norms (see \cite{V803}). Let $\mathcal{H}$ denotes any
of such space. As it was proved in \cite{V803} there exists a linear
isomorphism between the conjugate space $\mathcal{H}^{\ast}$ and the set $X_{%
\mathcal{H}}$ of all measurable functions $x$ (classes of $\mu$-equivalent
functions) for which%
\begin{equation*}
\int\left\vert x\varphi\right\vert \,d\mu<\infty\;\forall\varphi\in \mathcal{%
H}. 
\end{equation*}
This isomorphism has the form%
\begin{equation*}
X_{\mathcal{H}}\,\backepsilon\,x\longleftrightarrow f\left( \cdot\right)
=\int x\cdot\,d\mu\in\mathcal{H}^{\ast}. 
\end{equation*}
Using this fact, we can identify $\mathcal{H}^{\ast}$ with $X_{\mathcal{H}}$%
. Particularly,%
\begin{gather*}
L_{p}^{\ast}\equiv X_{L_{p}}=L_{p^{\prime}},\;N_{p}^{\ast}\equiv
X_{N_{p}}=N_{p^{\prime}}, \\
L_{2}^{\ast}\left( \Pi\right) \equiv X_{L_{2}\left( \Pi\right) }=L_{2}\left(
\Pi\right) ,
\end{gather*}
where $p>1$, $p^{\prime}=p/\left( p-1\right) .$

\subsection{Operators}

Let $w\left( t,\omega\right) $, $t\in\left[ 0,1\right] $ be a standard
Wiener process concurrent with the stream $\left\{ \mathfrak{F}_{t}\right\} $%
, having for $s>t$ independent of $\mathfrak{F}_{t}$ increments $w\left(
s\right) -w\left( t\right) ;$ let $\nu\left( t,A\right) $ be a random
Poisson measure on $R^{l+1}$with the parameter $\Pi\left( A\right) t$ ($%
\Pi\left( \cdot\right) $ is a $\sigma$-finite measure on $R^{l}$) concurrent
with $\left\{ \mathfrak{F}_{t}\right\} $, having for $s>t$ independent of $%
\mathfrak{F}_{t}$ increments $\nu\left( s,A\right) -\nu\left( t,A\right) $;
let $\tilde{\nu}\left( t,A\right) \equiv \nu\left( t,A\right) -\Pi\left(
A\right) t$ be the centered Poisson random measure.

Let us consider the following linear operations:%
\begin{gather}
\mathcal{L}\left( f\right) \left( t,\omega\right) \equiv\int
\limits_{0}^{t}f\left( s,\omega\right) \,ds,\;f\in E_{1},  \notag \\
\mathcal{J}\left( \varphi\right) \left( t,\omega\right) \equiv
\int\limits_{0}^{t}\varphi\left( s,\omega\right) \,dw\left( s,\omega \right)
,\ \varphi\in E_{2},  \label{F80f1} \\
\mathcal{P}\left( \alpha\right) \left( t,\omega\right) \equiv
\int\limits_{0}^{t}\int\alpha\left( s,\omega,y\right) \,\tilde{\nu}\left(
ds,dy\right) ,\ \alpha\in E_{2}\left( \Pi\right) .  \notag
\end{gather}
Integrals in (\ref{F80f1}) are interpreted respectively as the Lebesgue
integral over trajectories of random functions, the stochastic Ito integral
and the stochastic integral with respect to Poisson measure. Below we study
the operators generated by operations (\ref{F80f1}) on the above declared
spaces of random functions.

\subsection{Some preliminary facts}

Let us review some facts from the theory of martingales and stochastic
integrals \cite{V804}. Using the completeness of $\left( \Omega ,\mathfrak{F}%
,\mathbf{P}\right) $, we shall consider only separable modifications of
stochastic processes. Let $\mu\left( t\right) ,t\in\left[ 0,1\right] $ be a
martingale with respect to the stream $\left\{ \mathfrak{F}_{t}\right\} $%
\footnote{%
As all other martingales in this paper.}. As the stream is right-continuous,
then the process $\mu\left( t\right) $ has a modification with its
trajectories from $D$. If $\mathbf{M}\left\vert \mu\left( t\right)
\right\vert ^{p}<\infty$, $p\in\left( 1,\infty\right) $, then%
\begin{equation}
\mathbf{M}\sup_{0\leqslant t\leqslant1}\left\vert \mu\left( t\right)
\right\vert ^{p}\leqslant\left( p^{\prime}\right) ^{p}\sup_{0\leqslant
t\leqslant1}\mathbf{M}\left\vert \mu\left( t\right) \right\vert ^{p}=\left(
p^{\prime}\right) ^{p}\mathbf{M}\left\vert \mu\left( 1\right) \right\vert
^{p}.   \label{V80f2}
\end{equation}

Let $\mathcal{M}_{2}$ be the class of all martingales on $\left[ 0,1\right] $
with $\mathbf{M}\mu ^{2}\left( 1\right) <\infty $; $\mathcal{M}_{2}$ is a
Hilbert space with the inner product 
\begin{equation*}
\left( \mu _{1},\mu _{2}\right) =\mathbf{M}\mu _{1}\left( 1\right) \mu
_{2}\left( 1\right) .
\end{equation*}%
A closed subspaces in $\mathcal{M}_{2}$ form the sets of martingales of the
following type%
\begin{gather*}
\mathcal{M}_{2}^{w}=\left\{ \int\limits_{0}^{t}\varphi \,dw,\;\varphi \in
L_{2}\right\} , \\
\mathcal{M}_{2}^{\nu }=\left\{ \int\limits_{0}^{t}\int \psi \tilde{\nu}%
\left( \,ds,dy\right) ,\;\psi \in L_{2}\left( \Pi \right) \right\} .
\end{gather*}

Let $\varphi \in N_{p}$, $p\in \left( 1,\infty \right) ,\psi \in L_{2}\left(
\Pi \right) $. The following inequalities for the stochastic integrals are
well known:%
\begin{gather}
A_{p}^{p}\left\Vert \varphi \right\Vert _{N_{p}}^{p}\leqslant \mathbf{M}%
\left\vert \int_{0}^{1}\varphi \left( s\right) \,dw\left( s\right)
\right\vert ^{p}\leqslant B_{p}^{p}\left\Vert \varphi \right\Vert
_{N_{p}}^{p},  \label{V80f3} \\
\mathbf{M}\left\vert \int\limits_{0}^{1}\int \psi \left( s,y\right) \tilde{%
\nu}\left( \,ds,dy\right) \right\vert ^{2}=\mathbf{M}\int\limits_{0}^{1}\int
\psi ^{2}\left( s,y\right) \,ds\Pi \left( dy\right) .  \label{V80f4}
\end{gather}%
Here $A_{p},B_{p}$ are the positive values depending only on $p$, $%
A_{2}=B_{2}=1$. According to \cite{Novikov} $A_{p}$ has the following
property: $\forall r>2$ there exists $r_{1}<r$ such that $\tilde{A}%
_{r}\equiv \inf\limits_{p\in \left( r_{1},r\right) }A_{p}>0$. The right one
of the inequalities (\ref{V80f3}) holds for $p\in \left( 0,1\right] $ also.
Moreover, if $\varphi \in N_{1}$, then the random function $\mathcal{J}%
\left( \varphi \right) \left( t,\omega \right) $ is a martingale and $%
\mathbf{M}\mathcal{J}\left( \varphi \right) \left( t,\omega \right) =0$.

Let $\mathfrak{F}_{t}^{w}\equiv\overline{\sigma\left[ w\left( s\right)
,s\leqslant t\right] }$. Denote by $N_{p}\left( w\right) $ a subspace of $%
N_{p}$ of all random functions progressively measurable with respect to the
stream $\left\{ \mathfrak{F}_{t}^{w}\right\} $. If a random value $\xi\left(
\omega\right) $ is $\mathfrak{F}_{1}^{w}$-measurable and belongs to $%
L_{2}\left( \Omega\right) $ then there exists a function $\lambda\in
N_{2}\left( w\right) $ such that 
\begin{equation}
\xi\left( \omega\right) =\mathbf{M}\xi\left( \omega\right) +\int
\limits_{0}^{1}\lambda\left( s,\omega\right) \,dw\left( s,\omega\right) \ 
\text{~}a.s.   \label{V80f5}
\end{equation}
(Ito-Clark theorem). Using inequalities (\ref{V80f3}) we generalize formula (%
\ref{V80f5}) for a more wide class of random values.

\begin{theorem}
\label{T1} Let a random value $\xi\left( \omega\right) $ is\ $F_{1}^{w}$%
-measurable and $M\left\vert \xi\right\vert ^{p}<\infty$, where $p\in\left(
1,\infty\right) $. Then there exit a function $\lambda\in N_{p}\left(
w\right) $ such that the equality (\ref{V80f5}) is fulfilled. Moreover,%
\begin{equation}
A_{p}\left\Vert \lambda\right\Vert _{N_{p}}\leqslant\left\Vert \xi -\mathbf{M%
}\xi\right\Vert _{p}\leqslant B_{p}\left\Vert \lambda\right\Vert _{N_{p}}. 
\label{AB}
\end{equation}
In particular case, if a martingale $\mu\left( t\right) $ is concurrent with
the stream$\left\{ \mathfrak{F}_{t}^{w}\right\} $ and $M\left\vert \mu\left(
1\right) \right\vert ^{p}<\infty$, then $\mu\left( t\right) $ can be
identically represented in the form 
\begin{equation}
\mu\left( t\right) =\mathbf{M}\mu\left( t\right)
+\int\limits_{0}^{t}\lambda\left( s,\omega\right) \,dw\left( s,\omega\right)
,\;\lambda\in N_{p}\left( w\right) .   \label{V80f6}
\end{equation}
\end{theorem}

\begin{proof}
\textsf{\ }Denote by $P\left( p\right) $ the statement claiming that the
theorem holds for a given $p$. First prove the following two statements.

\begin{lemma}
\label{PP}$P\left( p\right) \Rightarrow P\left( r\right) $ $\forall
r\in\left( 1,p\right) .$
\end{lemma}

\begin{proof}
Let us use the fact that $L_{p}\left( w\right) $ is everywhere dense in $%
L_{r}\left( w\right) $ for $p>r$. Let $\xi\in L_{r}\left( w\right) $, $%
\left\{ \xi_{n}\right\} \subset L_{p}\left( w\right) $ and $\left\Vert
\xi-\xi_{n}\right\Vert _{r}\rightarrow0$ as $n\rightarrow\infty.$ Also let $%
\left\{ \lambda_{n}\right\} \subset N_{p}\left( w\right) \subset$ $%
N_{r}\left( w\right) $ be a sequence of functions from (\ref{V80f5})
corresponding to the sequence $\left\{ \xi_{n}\right\} $. According to (\ref%
{V80f3}) 
\begin{equation*}
A_{r}\left\Vert \lambda_{n}-\lambda_{m}\right\Vert
_{N_{r}}\leqslant\left\Vert \xi_{n}-\mathbf{M}\xi_{n}-\xi_{m}+\mathbf{M}%
\xi_{m}\right\Vert _{r}\leqslant2\left\Vert \xi_{n}-\xi_{m}\right\Vert
_{r}\rightarrow0 
\end{equation*}
as $n,m\rightarrow\infty$. We see that the sequence $\left\{ \lambda
_{n}\right\} $ is fundamental in $N_{r}\left( w\right) $; therefore, it
converges to some element $\lambda\in N_{r}\left( w\right) $. Let $\tilde{\xi%
}\equiv\mathbf{M}\xi+\mathcal{J}\left( \lambda\right) \left( 1\right) $.
Using the right inequality in (\ref{V80f3}), we get 
\begin{gather*}
\left\Vert \xi-\tilde{\xi}\right\Vert _{r}\leqslant\left\Vert \xi-\xi
_{n}\right\Vert _{r}+\left\Vert \xi_{n}-\tilde{\xi}\right\Vert _{r}, \\
\left\Vert \xi_{n}-\tilde{\xi}\right\Vert _{r}\leqslant\left\Vert \mathbf{M}%
\left( \xi_{n}-\tilde{\xi}\right) \right\Vert _{r}+\left\Vert
\int\limits_{0}^{1}\left( \lambda_{n}-\lambda\right) dw\right\Vert
_{r}\leqslant \\
\leqslant\left\Vert \xi_{n}-\xi\right\Vert _{r}+B_{r}\left\Vert \lambda
_{n}-\lambda\right\Vert _{N_{r}}.
\end{gather*}
Taking limits in these inequalities as $n\rightarrow\infty$, we receive $\xi=%
\tilde{\xi}$ a.s. That proves Lemma \ref{PP}. The uniqueness of the
representation (\ref{V80f5}) can be easily received from the left of
inequalities in (\ref{V80f3})
\end{proof}

\begin{lemma}
\textit{\ \label{Pr}Let }$r\geqslant4$\textit{\ and for every }$\xi\in
L_{r}\left( w\right) $\textit{\ there exists a function }$\lambda\in
N_{r-2}\left( w\right) $\textit{\ such that (\ref{V80f5}) holds. Then }$%
P\left( r\right) $\textit{.}
\end{lemma}

\begin{proof}
Let take the martingale $\eta\left( t\right) =\mathcal{J}\left(
\lambda\right) \left( t\right) $, $\,\eta\left( 1\right) =\xi$. Applying
formula Ito, we obtain for all $t\in\left[ 0,1\right] $ 
\begin{align}
\eta^{2}\left( t\right) & =\int\limits_{0}^{t}\lambda^{2}\,ds+2\int
\limits_{0}^{t}\eta\lambda\,dw,  \label{eta2} \\
\eta^{4}\left( t\right) & =6\int\limits_{0}^{t}\eta^{2}\lambda
^{2}\,ds+4\int\limits_{0}^{t}\eta^{3}\lambda\,dw\ \ a.s.   \label{eta4}
\end{align}
Prove that $\lambda\in N_{p_{n}}$, $p_{n}=r-2^{-n}$, $n=-1,0,1,2,\ldots$ For 
$n=-1$ this statement follows from the assumption of Lemma \ref{Pr}. Let us
prove that from realizability of the lemma for $n=k-1$, $k\geq0$ the
realizability of it for $n=k$ follows. Then, using the induction on $n$ we
shall prove Lemma \ref{Pr}.

Using H\"{o}lder inequality and (\ref{V80f2}), we receive%
\begin{gather*}
\mathbf{M}\left\vert \int\limits_{0}^{t}\eta^{6}\lambda^{2}\,ds\right\vert
^{p_{k}/8}\leqslant\mathbf{M}\sup\limits_{0\leqslant t\leqslant1}\left\vert
\eta\left( t\right) \right\vert ^{3p_{k}/4}\left(
\int\limits_{0}^{t}\lambda^{2}\,ds\right) ^{p_{k}/8}\leqslant \\
\leqslant\left( \frac{d_{k}}{d_{k}-1}\right) ^{3p_{k}/4}\left( \mathbf{M}%
\left\vert \xi\right\vert ^{d_{k}}\right) ^{3p_{k}/4d_{k}}\left( \left( 
\mathbf{M}\int\limits_{0}^{t}\lambda^{2}\,ds\right) ^{p_{k-1}/2}\right)
^{p_{k}/4p_{k-1}}<\infty, \\
d_{k}\equiv\frac{3p_{k}p_{k-1}}{4p_{k-1}-p_{k}}\leqslant r,\;k=0,1,2\ldots
\end{gather*}
Thus, due to (\ref{V80f3}) and (\ref{eta4}) we get 
\begin{equation*}
\mathbf{M}\left\vert \int\limits_{0}^{1}\eta^{3}\lambda dw\right\vert
^{p_{k}/4}\leqslant B_{p_{k}/4}^{p_{k}/4}\mathbf{M}\left(
\int\limits_{0}^{1}\eta^{6}\lambda^{2}ds\right) ^{p_{k}/8}, 
\end{equation*}%
\begin{equation*}
\mathbf{M}\left\vert 6\int\limits_{0}^{1}\eta^{2}\lambda^{2}ds\right\vert
^{p_{k}/4}\leqslant\mathbf{M}\left\vert \xi^{4}-4\int\limits_{0}^{1}\eta
^{3}\lambda dw\right\vert ^{p_{k}/4}<\infty. 
\end{equation*}
Next using (\ref{eta2}) and (\ref{V80f3}), we get%
\begin{equation*}
\begin{array}{c}
\mathbf{M}\left\vert \int_{0}^{1}\lambda^{2}ds\right\vert ^{p_{k}/2}\leqslant
\\ 
\leqslant2^{p_{k}/2-1}\left[ \mathbf{M}\left\vert \xi\right\vert
^{p_{k}}+2^{p_{k}/2}B_{p_{k}/2}^{p_{k}/2}\mathbf{M}\left( \int_{0}^{1}\eta
^{2}\lambda^{2}ds\right) ^{p_{k}/4}\right] <\infty,%
\end{array}
\end{equation*}
that proves the statement. Therefore, $\lambda\in N_{p_{n}}$ and according
to (\ref{V80f3}) 
\begin{equation*}
A_{p_{n}}\left\Vert \lambda\right\Vert _{N_{p_{n}}}\leqslant2\left\Vert
\xi\right\Vert _{p_{n}}\leqslant2\left\Vert \xi\right\Vert _{r}. 
\end{equation*}
Taking in the left-hand side of the equality the limit as $n\rightarrow\infty
$ and using of lower continuity of function $p\rightarrow\left\Vert
\lambda\right\Vert _{N_{p}}$\cite{V807}, we obtain 
\begin{equation*}
\tilde{A}_{r}\left\Vert \lambda\right\Vert _{N_{r}}\leqslant2\left\Vert
\xi\right\Vert _{r}, 
\end{equation*}
so, $\lambda\in N_{r}$ . This proves Lemma \ref{Pr}
\end{proof}

Using of the lemmas, we can now prove Theorem \ref{T1}. Really, from Lemma %
\ref{Pr} we have: $P\left( p\right) \Rightarrow P\left( p+2\right) $ $%
\forall p\geqslant4.$ But $P\left( 4\right) $ follows from Ito-Clark theorem
(\ref{V80f5}). So, for all $n=2,3,\ldots$ the statement $P\left( 2n\right) $
holds. Using of Lemma \ref{PP}, we receive $P\left( r\right) $ for all $%
r\in\left( 1,2n\right) $ that is equivalent to the first statement of
Theorem \ref{T1}. For receiving of the inequalities (\ref{AB}) it is
sufficiently to apply the inequalities (\ref{V80f3}) to (\ref{V80f5}). To
get (\ref{V80f6}) it is sufficiently to represent $\mu\left( 1\right) $ in
the form (\ref{V80f5}). Then%
\begin{gather*}
\mu\left( t\right) =\mathbf{M}\left\{ \left. \mu\left( 1\right) \right\vert 
\mathfrak{F}_{t}\right\} =\mathbf{M}\mu\left( 1\right) +\mathbf{M}\left\{
\left. \int\limits_{0}^{1}\varphi\,\left( s\right) dw\left( s\right)
\right\vert \mathfrak{F}_{t}\right\} = \\
=\mathbf{M}\mu\left( 1\right) +\int\limits_{0}^{t}\varphi\,\left( s\right)
dw\left( s\right) .
\end{gather*}
This concludes the proof of Theorem \ref{T1}
\end{proof}

\begin{corollary}
\textit{\ Let }$\mathbf{M}\left\vert \xi\right\vert ^{r}<\infty$\textit{\
and }$\mathbf{M}\left\vert \xi\right\vert
^{r+\epsilon}=\infty\forall\epsilon >0$\textit{. Then in (\ref{V80f5}) the
function }$\lambda\in N_{r}\left( w\right) $\textit{\ but }$\lambda\not \in
N_{r+\epsilon}\left( w\right) $\textit{, }$\epsilon>0$.
\end{corollary}

\begin{corollary}
\textit{Let }$\mu\left( t\right) =\mathcal{J}\left( \lambda\right) \left(
t\right) $,\textit{\ where }$\lambda\in E_{2}$\textit{, the process }$%
\mu\left( t\right) $ is concurrent with the stream\textit{\ }$\left\{ 
\mathfrak{F}_{t}^{w}\right\} $\textit{\ and }$\mathbf{M}\sup
\limits_{0\leqslant t\leqslant1}\left\vert \eta\left( t\right) \right\vert
^{p}<\infty$\textit{\ for some }$p\in\left( 1,\infty\right) $\textit{. Then }%
$\mu\left( t\right) $ is a martingale and $\left\Vert \lambda\right\Vert
_{N_{p}}<\infty$\textit{.}
\end{corollary}

\begin{lemma}
\label{L2}\textsf{\ }\textit{Let }$\varphi,\varkappa\in N_{p}$\textit{, }$p>1
$\textit{. Then}%
\begin{equation*}
\mathbf{M}\int\limits_{0}^{t}\varphi\,\left( s\right) dw\left( s\right)
\int\limits_{0}^{t}\varkappa\,\left( s\right) dw\left( s\right) =\mathbf{M}%
\int\limits_{0}^{t}\varkappa\,\left( s\right) \varphi\,\left( s\right)
\,ds,\;t\in\left[ 0,1\right] . 
\end{equation*}
\end{lemma}

\begin{proof}
Denote by 
\begin{align*}
\eta_{1}\left( t\right) & \equiv\int\limits_{0}^{t}\varphi\,\left( s\right)
dw\left( s\right) ,\; \\
\eta_{2}\left( t\right) & \equiv\int\limits_{0}^{t}\varkappa\,\left(
s\right) dw\left( s\right) .
\end{align*}
Using Ito's formula we have%
\begin{equation*}
\eta_{1}\left( t\right) \eta_{2}\left( t\right)
=\int\limits_{0}^{t}\eta_{1}\varkappa\,dw+\int\limits_{0}^{t}\eta_{2}\varphi
dw+\int \limits_{0}^{t}\varphi\varkappa ds. 
\end{equation*}
Now we must only prove that%
\begin{equation}
\mathbf{M}\int\limits_{0}^{t}\eta_{1}\varkappa\,dw=\mathbf{M}\int
\limits_{0}^{t}\eta_{2}\varphi\,dw=0.   \label{V80f8}
\end{equation}
Applying to $\eta_{1}\left( t\right) $ the inequalities (\ref{V80f2}), (\ref%
{V80f3}) we receive%
\begin{equation*}
\mathbf{M}\sup_{0\leqslant t\leqslant1}\left\vert \eta_{1}\left( t\right)
\right\vert ^{p}\leqslant\left( p^{\prime}\right) ^{p}\mathbf{M}\left\vert
\eta\left( 1\right) \right\vert ^{p}\leqslant\left( p^{\prime}\right)
^{p}B_{p}^{p}\left\Vert \varphi\right\Vert _{N_{p}}^{p}<\infty. 
\end{equation*}
Hence, 
\begin{gather*}
M\left( \int_{0}^{1}\left\vert \eta_{1}\varkappa\right\vert ^{2}ds\right)
^{1/2}\leqslant\mathbf{M}\sup_{0\leqslant t\leqslant1}\left\vert \eta
_{1}\left( t\right) \right\vert \left( \int_{0}^{1}\left\vert
\varkappa\right\vert ^{2}ds\right) ^{1/2}\leqslant \\
\leqslant\left( \mathbf{M}\sup_{0\leqslant t\leqslant1}\left\vert \eta
_{1}\left( t\right) \right\vert ^{p}\right) ^{1/p}\left\Vert \varkappa
\right\Vert _{N_{p^{\prime}}}<\infty,
\end{gather*}
so, $\eta_{1}\varkappa\in N_{1}$. Analogously, $\eta_{2}\varphi\in N_{1}$.
This proves (\ref{V80f8}) and the lemma
\end{proof}

\section{Operator\ $\mathcal{L}$}

\subsection{Definition of operator $\mathcal{L}$}

Let $f\in L_{1}$. Random functions $\mathcal{L}\left( \tilde{f}\right)
\left( t,\omega\right) $ corresponding to different members $\tilde{f}$ of
the class $f$ are continuous a.s. and indistinguishable from the each other.
So, we can define the operator $\mathcal{L}:$ $L_{p}\rightarrow CH_{p}.$

\begin{lemma}
\label{L} \textit{The linear operator }$\mathcal{L}:$\textit{\ }$%
L_{p}\rightarrow CH_{p}$\textit{, }$p\in\left[ 1,\infty\right) $\textit{\ is
bounded.}
\end{lemma}

\begin{proof}
Using of H\"{o}lder inequality, we get the inequality%
\begin{equation*}
\sup_{0\leqslant t\leqslant1}\left\vert \int\limits_{0}^{t}f\left( s\right)
\,ds\right\vert ^{p}\leqslant\int\limits_{0}^{1}\left\vert f\left( s\right)
\right\vert ^{p}\,ds. 
\end{equation*}
By averaging of it by the measure $\mathbf{P}$ we then receive%
\begin{equation*}
\left\vert \mathcal{L}\left( f\right) \right\vert _{p}^{\left( d\right)
}\leqslant\left\Vert f\right\Vert _{p}
\end{equation*}
that completes the proof
\end{proof}

\subsection{Conjugate operator $\mathcal{L}^{\ast}$}

Denote by $\mathcal{L}^{\ast}:$ $\left( CH_{p}\right) ^{\ast}\rightarrow
L_{p}^{\ast}$ conjugate to the operator $\mathcal{L}$. We'll obtain the
explicit representation for the restriction of $\mathcal{L}^{\ast}$ on a
subspace $L_{p^{\prime}}\equiv\left( L_{p}\right) ^{\ast}$ of the space $%
\left( CH_{p}\right) ^{\ast}$.

\begin{theorem}
\textit{Let }$f\in L_{p^{\prime}}$\textit{. Then there exist selections of
conditional expectations such that}%
\begin{equation}
\mathcal{L}^{\ast}\left( f\right) \left( t,\omega\right) =\mathbf{M}\left\{
\left. \int\limits_{t}^{1}f\left( s,\omega\right) \,ds\right\vert \mathfrak{F%
}_{t}\right\} .   \label{V80f9}
\end{equation}
\textit{More precisely: all }$BF$\textit{-measurable functions from the
right-hand side of (\ref{V80f9}) belong to the class }$L^{\ast}\left(
f\right) $\textit{.}
\end{theorem}

\begin{proof}
In accordance to definition of the conjugate operator $\mathcal{L}^{\ast
}\left( f\right) $ is an element $h$ of $L_{p^{\prime}}$ such that for all $%
\varphi\in L_{p}$%
\begin{equation*}
\left\langle f,\mathcal{L}\left( \varphi\right) \right\rangle =\left\langle
h,\varphi\right\rangle =\mathbf{M}\int\limits_{0}^{1}h\left( s\right)
\,\varphi\left( s\right) \,ds. 
\end{equation*}
As the random function $\int_{t}^{1}f\left( s,\omega\right) \,ds$ is $%
\mathfrak{B}^{1}\times\mathfrak{F}$-measurable (Fubini theorem), then (see 
\cite{V805}) there exist selections of conditional expectations such that
the random function $\tilde{h}\left( t\right) $ from right-hand side of (\ref%
{V80f9}) is $\mathfrak{BF}$-measurable. It is easy to prove that $\tilde{h}%
\in L_{p^{\prime}}$. Using Fubini theorem we get for any $\varphi\in L_{p}$%
\begin{gather*}
\mathbf{M}\int\limits_{0}^{1}\tilde{h}\left( s\right) \,\varphi\left(
s\right) \,ds=\int\limits_{0}^{1}\mathbf{MM}\left\{ \left. \int
\limits_{s}^{1}f\left( t,\omega\right) \,dt\right\vert \mathfrak{F}%
_{s}\right\} \,\varphi\left( s\right) \,ds= \\
=\int\limits_{0}^{1}\mathbf{MM}\left\{ \left. \int\limits_{s}^{1}f\left(
t,\omega\right) \,dt\varphi\left( s\right) \,\right\vert \mathfrak{F}%
_{s}\right\} \,ds= \\
=\mathbf{M}\int\limits_{0}^{1}\int\limits_{s}^{1}f\left( t,\omega\right)
\,dt\varphi\left( s\right) \,ds=\left\langle f,\mathcal{L}\left(
\varphi\right) \right\rangle .
\end{gather*}
This completes the proof of the theorem
\end{proof}

\begin{corollary}
\textit{The restriction of the operator }$L^{\ast}$\textit{\ on }$L_{q}$%
\textit{, }$q\in\left( 1,\infty\right) $ \textit{is bounded as an operator
from }$L_{q}$\textit{\ to }$DH_{q}$\textit{.}
\end{corollary}

\begin{proof}
Using (\ref{V80f9}), we get for any $f\in L_{q}$%
\begin{align}
\mathcal{L}^{\ast}\left( f\right) \left( t,\omega\right) & =\mathbf{M}%
\left\{ \left. \int\limits_{0}^{1}f\left( s,\omega\right) \,ds\right\vert 
\mathfrak{F}_{t}\right\} -\int\limits_{0}^{t}f\left( s,\omega\right)
\,ds\equiv  \label{V80f10} \\
& \equiv\mu\left( t,\omega\right) -\mathcal{L}\left( f\right) \left(
t,\omega\right) .  \notag
\end{align}
As a martingale the process $\mu\left( t\right) $ has a modification with
trajectories from $D$. Using (\ref{V80f2}), we get%
\begin{equation*}
\mathbf{M}\sup_{0\leqslant t\leqslant1}\left\vert \mu\left( t\right)
\right\vert ^{q}\leqslant\left( q^{\prime}\right) ^{q}\mathbf{M}\left\vert
\mu\left( 1\right) \right\vert ^{q}\leqslant\left( q^{\prime}\right)
^{q}\left\Vert f\right\Vert _{q}^{q}. 
\end{equation*}
Taking into account the inequalities and the properties of operator $%
\mathcal{L}$ also we obtain the corollary from (\ref{V80f10})
\end{proof}

\begin{corollary}
\textit{Suppose that the stream }$\left\{ \mathfrak{F}_{t}\right\} $\textit{%
\ is continuous in }$\left[ 0,1\right] $\textit{. Then the restriction of }$%
L^{\ast}$\textit{\ on }$L_{q}$\textit{, }$q\in\left( 1,\infty\right) $%
\textit{\ is bounded as an operator from }$L_{q}$\textit{\ to }$CH_{q}$%
\textit{.}
\end{corollary}

\begin{proof}
As in this case the martingale $\mu\left( t\right) $ has a continuous
modification, then the corollary leads immediately from (\ref{V80f10})
\end{proof}

\section{Operator\ $\mathcal{J}$}

\subsection{Definition of $\mathcal{J}$}

Let $\varphi\in N_{1}$. We consider only continuous modifications of Ito
integral. In this case random functions $\mathcal{J}\left( \tilde{\varphi }%
\right) \left( t,\omega\right) $ corresponding to different members $\tilde{%
\varphi}$ of the class $\varphi$ are indistinguishable from the each other.
So, we can define the operator $\mathcal{J}:N_{p}\rightarrow CH_{p}$.

\begin{lemma}
\textit{The linear operator }$\mathcal{J}:N_{p}\rightarrow CH_{p}$\textit{,}$%
\;p\in\left[ 2,\infty\right) $ \textit{is bounded.}
\end{lemma}

\begin{proof}
The random function $\mathcal{J}\left( \varphi\right) \left( t,\omega
\right) $ is a martingale for any $\varphi\in N_{p}$. Using (\ref{V80f2})
and (\ref{V80f3}), we get%
\begin{equation*}
\mathbf{M}\sup_{0\leqslant t\leqslant1}\left\vert \mathcal{J}\left(
\varphi\right) \left( t,\omega\right) \right\vert ^{p}\leqslant\left(
p^{\prime}\right) ^{p}\mathbf{M}\left\vert \mathcal{J}\left( \varphi\right)
\left( 1,\omega\right) \right\vert ^{p}\leqslant\left( p^{\prime}\right)
^{p}B_{p}^{p}\left\Vert \varphi\right\Vert _{N_{p}}^{p}. 
\end{equation*}
This inequality can be rewrite in the form%
\begin{equation*}
\left\vert \mathcal{J}\left( \varphi\right) \left( t,\omega\right)
\right\vert _{p}^{\left( d\right) }\leqslant p^{\prime}B_{p}\left\Vert
\varphi\right\Vert _{N_{p}}
\end{equation*}
that proves the lemma
\end{proof}

\begin{corollary}
\textit{The operator }$\mathcal{J}:L_{p}\rightarrow CH_{p}$,$\;p\in\left[
2,\infty\right) $\textit{\ is bounded.}
\end{corollary}

\subsection{Operator $\mathcal{\tilde{J}}$}

Let $\mathcal{J}^{\ast}:\left( CH_{p}\right) ^{\ast}\rightarrow N_{p}^{\ast }
$ be conjugate to the operator $\mathcal{J}$. Let obtain the explicit
representation for the restriction of $\mathcal{J}^{\ast}$ to a subspace $%
L_{p^{\prime}}\equiv\left( L_{p}\right) ^{\ast}\subset\left( CH_{p}\right)
^{\ast}$. For this purpose let firstly perform some auxiliary constructions.

Let $N_{r}\left( \mathcal{B}\times \mathcal{BF}\right) $, $\,r\in \left(
1,\infty \right) $ be the Banach space of $\mathcal{B}\times \mathcal{BF}$%
-measurable functions $\lambda \left( t,s,\omega \right) $, $0\leqslant
t\leqslant 1$, $0\leqslant s\leqslant t$, $\omega \in \Omega $ (classes of
equivalent functions) with the norm%
\begin{equation*}
\left\vert \left\vert \left\vert \lambda \right\vert \right\vert \right\vert
_{r}=\left\{ \mathbf{M}\int\limits_{0}^{1}\left( \int\limits_{0}^{1}\lambda
^{2}\left( t,s,\omega \right) \right) ^{r/2}\right\} ^{1/r}.
\end{equation*}%
Consider a linear manifold $T_{r}$ in $N_{r}\left( \mathcal{B}\times 
\mathcal{BF}\right) $ consisting of functions of the type%
\begin{equation}
\lambda \left( t,s,\omega \right) =\sum\limits_{k=0}^{n-1}\chi _{\Delta
_{k}}\left( t\right) \alpha _{k}\left( s,\omega \right) ,\;t\in \left[ 0,1%
\right] ,~s\in \left[ 0,t\right] ,~\omega \in \Omega ,  \label{V80f11}
\end{equation}%
where $\Delta _{k}\equiv \left[ t_{k},t_{k-1}\right) $, $0=t_{0}<t_{1}<%
\ldots <t_{n}=1$, functions $\alpha _{k}\left( s,\omega \right) $ are $%
\mathcal{BF}$-measurable, $\alpha _{0}\equiv 0$, $\alpha _{k}\left( s,\omega
\right) =0$ for $s\in \left( t_{k},1\right] $, $k\geqslant 1$, 
\begin{equation*}
M\left( \int\limits_{0}^{t_{k}}\alpha _{k}^{2}\left( s,\omega \right)
\,ds\right) ^{r/2}<\infty ,
\end{equation*}%
$\chi _{A}\left( t\right) $ is the indicator of a set $A$.

Set for $\lambda\in T_{r}$%
\begin{equation*}
\mathcal{\tilde{J}}\left( \lambda\right) \left( t,\omega\right) \equiv
\sum\limits_{k=0}^{n-1}\chi_{\Delta_{k}}\left( t\right)
\int\limits_{0}^{t_{k}}\alpha_{k}\left( s,\omega\right) \,dw\left(
s,\omega\right) . 
\end{equation*}
It is easy to prove that the operator $\mathcal{\tilde{J}}$ is linear on the
manifold $T_{r}$. Applying inequalities (\ref{V80f3}), we then obtain the
inequalities%
\begin{equation}
A_{r}\left\vert \left\vert \left\vert \lambda\right\vert \right\vert
\right\vert _{r}\leqslant\left\Vert \mathcal{\tilde{J}}\left( \lambda\right)
\right\Vert _{r}\leqslant B_{r}\left\vert \left\vert \left\vert \lambda
\right\vert \right\vert \right\vert _{r}.   \label{V80f12}
\end{equation}

Let $\bar{T}_{r}$ be the closure of $T_{r}$ in $N_{r}\left( \mathcal{B}\times%
\mathcal{BF}\right) $. Using of the right inequality in (\ref{V80f12}), we
can perform the expansion of operator $\mathcal{\tilde{J}}$ on $\bar{T}_{r}$
by continuity; this expansion also keep the inequalities (\ref{V80f12}).

\begin{lemma}
\label{Lr}\textit{The set }$L_{r}\left( w\right) \equiv\left\{ \mathcal{%
\tilde{J}}\left( \lambda\right) ,\;\lambda\in\bar{T}_{r}\right\} $\textit{\
is closed in }$L_{r}.$
\end{lemma}

\begin{proof}
\textsf{\ }Let $\left\{ \xi_{n}=\mathcal{\tilde{J}}\left( \lambda_{n}\right)
,\;\lambda_{n}\in\bar{T}_{r}\right\} \subset L_{r}\left( w\right) $ and $%
\left\Vert \xi_{n}-\xi_{m}\right\Vert _{r}\rightarrow0$ as $n,m\rightarrow
\infty$. Using the left inequality in (\ref{V80f12}) we obtain 
\begin{equation*}
A_{r}\left\vert \left\vert \left\vert \lambda_{n}-\lambda_{m}\right\vert
\right\vert \right\vert _{r}\leqslant\left\Vert \xi_{n}-\xi_{m}\right\Vert
_{r}\rightarrow0,\;n,m\rightarrow\infty. 
\end{equation*}
The sequence $\left\{ \lambda_{n}\right\} $ is fundamental in $N_{r}\left( 
\mathcal{B}\times\mathcal{BF}\right) $, so it converges to some element $%
\lambda\in N_{r}\left( \mathcal{B}\times\mathcal{BF}\right) $. As $\bar {T}%
_{r}$ is closed, then $\lambda\in\bar{T}_{r}$. Let $\xi=\mathcal{\tilde{J}}%
\left( \lambda\right) .$ From the right inequality in (\ref{V80f12}) we get%
\begin{equation*}
\left\Vert \xi_{n}-\xi\right\Vert _{r}\leqslant B_{r}\left\vert \left\vert
\left\vert \lambda_{n}-\lambda\right\vert \right\vert \right\vert
_{r}\rightarrow0,\;n\rightarrow\infty, 
\end{equation*}
that proves the lemma
\end{proof}

Let study the operator $\mathcal{\tilde{J}}$.

\begin{lemma}
\label{ad}\textit{Let }$\lambda\in\bar{T}_{r}$\textit{, }$\eta=\mathcal{%
\tilde {J}}\left( \lambda\right) $\textit{. Then}
\end{lemma}

\begin{description}
\item[a)] $M\eta\left( t\right) =0$\textit{\ for almost all }$t\in\left[ 0,1%
\right] $\textit{\ ;}

\item[b)] $\forall t\in\left[ 0,1\right] $\textit{\ almost all }$s\in\left[
t,1\right] $\textit{\ a.s. }%
\begin{equation*}
\mathbf{M}\left\{ \left. \eta\left( s\right) \right\vert \mathcal{F}%
_{t}\right\} =\int\limits_{0}^{t}\lambda\left( t,s,\omega\right) \,dw\left(
s,\omega\right) ; 
\end{equation*}

\item[c)] \textit{if }$\varphi\in N_{r^{\prime}}$,\textit{\ then for almost
all }$t\in\left[ 0,1\right] $%
\begin{equation*}
\mathbf{M}\eta\left( t\right) \mathcal{J}\left( \varphi\right) \left(
t\right) =\mathbf{M}\int\limits_{0}^{t}\lambda\left( t,s,\omega\right)
\varphi\left( s,\omega\right) \,ds; 
\end{equation*}

\item[d)] \textit{if }$r=2$\textit{\ and }$\tau\in\left( 0,1\right] $,%
\textit{\ then }%
\begin{equation*}
\int\limits_{o}^{\tau}\eta\left( s\right) \,ds=\int\limits_{o}^{\tau}\left(
\int\limits_{t}^{\tau}\lambda\left( t,s,\omega\right) \,ds\right) \,dw\left(
t,\omega\right) ~\text{~}\mathit{a.s.}
\end{equation*}
\end{description}

\begin{proof}
Since $\lambda\in N_{r}\left( \mathcal{B}\times\mathcal{BF}\right) $, it
follows that $\lambda\left( t,\cdot\right) \in N_{r}\left( \left[ 0,t\right]
\times\Omega\right) $ for almost all $t\in\left[ 0,1\right] $. So, we can
define the value%
\begin{equation*}
\xi_{t}\left( \omega\right) =\int\limits_{0}^{t}\lambda\left(
t,s,\omega\right) \,dw\left( s,\omega\right) , 
\end{equation*}
where the integral in the right-hand side is of Ito type. Let us show that
for almost all $t$ 
\begin{equation}
\xi_{t}\left( \omega\right) =\eta\left( t,\omega\right) \ a.s. 
\label{V80f13}
\end{equation}

Indeed, we have $\eta=\lim\limits_{n\rightarrow\infty}\left( L_{r}\right) 
\mathcal{\tilde{J}}\left( \lambda_{n}\right) $, where $\left\{ \lambda
_{n}\right\} \subset T_{r}$. Thus, for almost all $t$ and for some sequence $%
n_{k}\rightarrow\infty$%
\begin{equation*}
\mathbf{M}\left\vert \eta_{n_{k}}\left( t\right) -\eta\left( t\right)
\right\vert ^{r}\rightarrow0,\;\eta_{k}=\mathcal{\tilde{J}}\left( \lambda
_{k}\right) . 
\end{equation*}
In other hand, for some subsequence of the sequence $n_{k}$%
\begin{equation*}
\mathbf{M}\left\vert \eta_{n_{k_{m}}}\left( t\right) -\xi\left( t\right)
\right\vert ^{r}\leqslant B_{r}^{r}\mathbf{M}\left(
\int\limits_{0}^{t}\left\vert \lambda\left( t,s\right)
-\lambda_{n_{k_{m}}}\left( t,s\right) \right\vert ^{2}\,ds\right)
^{r/2}\rightarrow0,\;m\rightarrow \infty 
\end{equation*}
for almost all $t$. Hence, for almost all $t\in\left[ 0,1\right] $ we have $%
\mathbf{M}\left\vert \eta\left( t\right) -\xi_{t}\right\vert ^{r}=0$ that
implies (\ref{V80f13}). Using this result, the Lemma \ref{ad} and Ito's
integral properties we obtain items (a)-(c) of the lemma.

Let prove the last item. Denote by $\alpha$ and $\beta$, respectively,
random values in right-hand and left-hand sides in item (d). Let also%
\begin{equation*}
\alpha_{n}\equiv\int\limits_{0}^{\tau}\eta_{n}\left( s\right) \,ds,\;\beta
_{n}\equiv\int\limits_{0}^{\tau}\left( \int\limits_{t}^{\tau}\lambda
_{n}\left( s,t\right) \,ds\right) \,dw\left( t\right) . 
\end{equation*}
Using the construction of $\lambda_{n}\in T$, it is easy to verify that%
\begin{equation*}
\alpha_{n}=\sum\limits_{k=0}^{n-1}\limfunc{mes}\left( \Delta_{k}\cap\left[
0,\tau\right] \right) \int\limits_{0}^{t_{k}}\alpha_{k}\left( s\right)
\,dw\left( s\right) =\beta_{n}\ \ a.s. 
\end{equation*}
Further, 
\begin{gather*}
\mathbf{M}\left\vert \alpha_{n}-\alpha\right\vert ^{2}\leqslant\left\Vert
\eta_{n}-\eta\right\Vert _{2}^{2}\rightarrow0, \\
\mathbf{M}\left\vert \beta_{n}-\beta\right\vert ^{2}=\mathbf{M}\int
\limits_{0}^{\tau}\left\vert \int\limits_{t}^{\tau}\left( \lambda_{n}\left(
s,t\right) -\lambda\left( s,t\right) \right) \,ds\right\vert
^{2}\,dt\leqslant \\
\leqslant\mathbf{M}\int\limits_{0}^{1}\int\limits_{0}^{t}\left\vert
\lambda_{n}\left( s,t\right) -\lambda\left( s,t\right) \right\vert
^{2}\,ds\,dt\equiv\left\vert \left\vert \left\vert \lambda_{n}-\lambda
\right\vert \right\vert \right\vert
_{2}^{2}\rightarrow0,\;n\rightarrow\infty.
\end{gather*}
So, we have $\mathbf{M}\left\vert \alpha-\beta\right\vert ^{2}=0$ that
completes the proof
\end{proof}

\begin{lemma}
\textit{For every }$\varphi\in L_{r}\left( w\right) $\textit{, }$r\in\left(
1,2\right] $\textit{\ there exist an element }$\lambda\in\bar{T}_{r}$\textit{%
\ such that }$\limfunc{mes}\times P$\textit{-a.s. }%
\begin{equation}
\varphi\left( t,\omega\right) =\mathbf{M}\varphi\left( t,\omega\right) +%
\mathcal{\tilde{J}}\left( \lambda\right) \left( t,\omega\right) . 
\label{V80f14}
\end{equation}
\textit{\ Linear operator }$\mathcal{K}:\lambda=\mathcal{K}\varphi$\textit{\
is bounded from }$L_{r}\left( w\right) $\textit{\ to }$N_{r}\left( \mathcal{B%
}\times\mathcal{BF}\right) $\textit{.}
\end{lemma}

\begin{proof}
Let\textsf{\ }$S_{r}$ is the set of step-functions in $L_{r}\left( w\right) $%
, i.e., the set of all functions of the type%
\begin{equation*}
\psi\left( t,\omega\right) =\sum\limits_{k=0}^{n-1}\chi_{\Delta_{k}}\left(
t\right) \psi_{k}\left( \omega\right) ,\;\psi_{k}\in L_{r}\left(
\Omega\right) , 
\end{equation*}
where the random values $\psi_{k}\left( \omega\right) ,k\geqslant0$ are
measurable with respect to $\sigma$-algebras $\mathcal{F}_{t_{k}}^{w}$. It
is well known that $\bar{S}_{r}=L_{r}\left( w\right) $ \cite{V804}. Let
prove that any $\psi\in S_{r}$ can be represented in the form (\ref{V80f14})
with some $\lambda\in T_{r}$. In fact, according to Theorem \ref{T1} the
random function $\psi_{k}\left( \omega\right) $ can be represented in the
form%
\begin{equation*}
\psi_{k}\left( \omega\right) =\mathbf{M}\psi_{k}\left( \omega\right)
+\int\limits_{0}^{t_{k}}\alpha_{k}\left( s,\omega\right) \,dw\left(
s,\omega\right) , 
\end{equation*}
where $M\left( \int_{0}^{t_{k}}\alpha_{k}^{2}\left( s\right) \,ds\right)
^{r/2}$. Extending functions $\alpha_{k}\left( s,\omega\right) $ by zero to $%
\left[ 0,1\right] \times\Omega$ we obtain%
\begin{equation*}
\psi\left( t,\omega\right) =\mathbf{M}\psi\left( t,\omega\right) +\mathcal{%
\tilde{J}}\left( \lambda\right) \left( t,\omega\right) , 
\end{equation*}
where the function $\lambda$ has the form (\ref{V80f11}). Further, let $%
\varphi$ be an arbitrary function from $L_{r}\left( w\right) $ and $%
\left\Vert \varphi-\psi_{n}\right\Vert _{r}\rightarrow0,\,n\rightarrow\infty 
$, where $\left\{ \psi_{n}\right\} \subset S_{r}.$ Let $\lambda _{n}=%
\mathcal{K}\psi_{n}$. Applying the left-hand bound in (\ref{V80f12}) and
taking into account that the sequence $\left\{ \psi_{n}\right\} $ is
fundamental we receive 
\begin{equation*}
A_{r}\left\vert \left\vert \left\vert \lambda_{n}-\lambda_{m}\right\vert
\right\vert \right\vert _{r}\leqslant\left\Vert \mathcal{\tilde{J}}\left(
\lambda_{n}-\lambda_{m}\right) \right\Vert _{r}\leqslant2\left\Vert \psi
_{n}-\psi_{m}\right\Vert _{r}\rightarrow0 
\end{equation*}
as $n,m\rightarrow\infty$. Hence, $\left\{ \lambda_{n}\right\} $ is a
fundamental sequence in $N_{r}\left( \mathcal{B}\times\mathcal{BF}\right) $,
so, it converges to some $\lambda\in\bar{T}_{r}.$ Let $\varphi_{1}=\mathbf{M}%
\varphi+\mathcal{\tilde{J}}\left( \lambda\right) $. We get%
\begin{gather*}
\left\Vert \varphi_{1}-\varphi\right\Vert _{r}\leqslant\left\Vert \varphi
-\psi_{n}\right\Vert _{r}+\left\Vert \varphi_{1}-\psi_{n}\right\Vert
_{r}\leqslant \\
\leqslant\left\Vert \varphi-\psi_{n}\right\Vert _{r}+\left\Vert \mathbf{M}%
\left( \varphi-\psi_{n}\right) \right\Vert _{r}+ \\
+\left\Vert \mathcal{\tilde{J}}\left( \lambda_{n}-\lambda\right) \right\Vert
_{r}\leqslant2\left\Vert \varphi-\psi_{n}\right\Vert _{r}+B_{r}\left\vert
\left\vert \left\vert \lambda_{n}-\lambda\right\vert \right\vert \right\vert
_{r}\rightarrow0,\;n\rightarrow\infty
\end{gather*}
that proves (\ref{V80f14}). Applying to (\ref{V80f14}) the left-hand bound
in (\ref{V80f12}) we obtain%
\begin{equation*}
A_{r}\left\vert \left\vert \left\vert \lambda\right\vert \right\vert
\right\vert _{r}\leqslant\left\Vert \mathcal{\tilde{J}}\left( \lambda\right)
\right\Vert _{r}=\left\Vert \varphi-\mathbf{M}\varphi\right\Vert
_{r}\leqslant2\left\Vert \varphi\right\Vert _{r}. 
\end{equation*}
This shows the uniqueness of the representation (\ref{V80f14}) and the
boundedness of operator $\mathcal{K}$. This completes the proof
\end{proof}

\subsection{Conjugate operator $\mathcal{J}^{\ast}$}

Let study the structure of operator $\mathcal{J}^{\ast}$. Let $L_{2}^{\bot
}\left( w\right) $ be the orthogonal complement of $L_{2}\left( w\right) $
to Hilbert space $L_{2}$: $L_{2}=L_{2}\left( w\right) \oplus L_{2}^{\bot
}\left( w\right) $.

\begin{theorem}
\textit{The restrictions of the operator }$\mathcal{J}^{\ast}:\left(
CH_{p}\right) ^{\ast}\rightarrow N_{p}^{\ast},\,p\in\left[ 2,\infty\right) $%
\textit{\ on the subspace }$L_{2}\subset\left( CH_{p}\right) ^{\ast}$\textit{%
\ and also on the subspace }$L_{p^{\prime}}\left( w\right) \subset\left(
CH_{p}\right) ^{\ast}$\textit{\ have the form}%
\begin{equation}
\mathcal{J}^{\ast}\left( \chi\right) \left( t,\omega\right) =\int
\limits_{t}^{1}\lambda\left( s,t,\omega\right) \,ds   \label{V80f15}
\end{equation}
\textit{where }$\chi=\chi_{1}+\mathcal{\tilde{J}}\left( \lambda\right) $%
\textit{, }$\chi_{1}\in L_{2}^{\bot}\left( w\right) $\textit{, }$\lambda \in%
\bar{T}_{2}$\textit{, }$\lambda=\mathcal{K}\chi$\textit{.}
\end{theorem}

\begin{proof}
Let first of all prove the theorem for $\chi\in L_{p^{\prime}}\left(
w\right) ,p>2$. Let $\varphi\in N_{p}$. Expanding the function $\chi$ into
the sum (\ref{V80f14}) and taking into account item (b) of Lemma \ref{ad} we
receive on Fubini theorem%
\begin{gather*}
\left\langle \chi,\mathcal{J}\left( \varphi\right) \right\rangle =\mathbf{M}%
\int\limits_{0}^{1}\left( \mathbf{M}\chi\left( t\right) +\mathcal{\tilde{J}}%
\left( \lambda\right) \left( t\right) \right) \mathcal{J}\left(
\varphi\right) \left( t\right) \,dt= \\
=\int\limits_{0}^{1}\mathbf{M}\int\limits_{0}^{t}\lambda\left( t,s\right)
\varphi\left( s\right) \,dsdt=\mathbf{M}\int\limits_{0}^{1}\left(
\int\limits_{t}^{1}\lambda\left( t,s\right) \,ds\right) \varphi\left(
t\right) dt=\left\langle \mathcal{J}^{\ast}\chi,\varphi\right\rangle .\,
\end{gather*}
As $\varphi$ is an arbitrary function, then this completes the proof of
theorem for $p>2$. Case $\chi\in L_{2}$ can be studied analogously. Note,
that if $\chi\in L_{2}\left( w\right) \subset L_{2}\cap L_{p^{\prime}}\left(
w\right) $, then in accordance to (\ref{V80f14}) we have $\chi_{1}=M\chi\in
L_{2}^{\perp}\left( w\right) \,$
\end{proof}

\begin{corollary}
\textit{The restrictions of the operators listed in the theorem are bounded
from }$L_{2}$\textit{\ to }$L_{2}$\textit{\ and from }$L_{q}\left( w\right) $%
\textit{\ to }$L_{q}\left( w\right) $, $q\in\left[ p^{\prime},2\right] $%
\textit{\ respectively.}
\end{corollary}

\begin{proof}
Really, we get from formula (\ref{V80f15}) 
\begin{align*}
\left\Vert \mathcal{J}^{\ast}\chi\right\Vert _{q}^{q} & =\mathbf{M}%
\int\limits_{0}^{1}\left\vert \int\limits_{t}^{1}\lambda\left( s,t\right)
\,ds\right\vert ^{q}dt\leqslant\mathbf{M}\int\limits_{0}^{1}\int
\limits_{t}^{1}\left\vert \lambda\left( s,t\right) \right\vert ^{q}dsdt= \\
& =\mathbf{M}\int\limits_{0}^{1}\int\limits_{0}^{t}\left\vert \lambda\left(
s,t\right) \right\vert ^{q}dsdt\leqslant\left\vert \left\vert \left\vert
\lambda\right\vert \right\vert \right\vert _{q}^{q}\leqslant\left(
2A_{q}\right) ^{-q}\left\Vert \chi\right\Vert _{q}^{q}.
\end{align*}
The first statement of the corollary can be proved analogously
\end{proof}

\section{Operator\ $\mathcal{P}$}

\subsection{Definition of $\mathcal{P}$}

Let $\psi\in L_{2}\left( \Pi\right) $. Since the random functions $\mathcal{P%
}\left( \tilde{\psi}\right) \left( t,\omega\right) $ corresponding to
different members $\tilde{\psi}$ of the class $\psi$ a.s. have trajectories
from $D$ and are indistinguishable from the each other, then it is possible
to define operators $\mathcal{P}:L_{p}\left( \Pi\right) \rightarrow
DH_{p},\;p\geqslant2$.

\begin{lemma}
\textit{Linear operator }$\mathcal{P}$ $:L_{p}\left( \Pi\right) \rightarrow
DH_{p},\;p\geqslant2$ \textit{( }$p$ is an even number\textit{) is bounded. }
\end{lemma}

\begin{proof}
For $p=2$ the lemma follows from (\ref{V80f4}). Let $p>2$. Prove the
boundedness of operator $\mathcal{P}$ $:L_{p}\left( \Pi\right) \rightarrow
DH_{p}$. Let use the estimations from \cite[p. 150]{V806}. For $\alpha\in
L_{p}\left( \Pi\right) $%
\begin{align*}
\mathbf{M}\left\vert \mathcal{P}\left( \alpha\right) \left( t\right)
\right\vert ^{p} & \leqslant K_{1}\left( p\right) \mathbf{M}\int
\limits_{0}^{t}\sum\limits_{k=2}^{p}v_{k}^{p/k}\left( s\right) \,ds, \\
v_{k}\left( s\right) & \equiv\int\left\vert \alpha\left( s,y\right)
\right\vert ^{k}\Pi\left( dy\right) .
\end{align*}
Using H\"{o}lder inequality and binomial formula of Newton we get%
\begin{gather*}
\mathbf{M}\left\vert \mathcal{P}\left( \alpha\right) \left( t\right)
\right\vert ^{p}\leqslant K_{2}\left( p\right) \mathbf{M}\sum%
\limits_{k=2}^{p}\left( \int\limits_{0}^{t}v_{2}^{p/2}\left( s\right)
\,ds\right) ^{\frac{2\left( p-k\right) }{k\left( p-2\right) }}\left( \int
\limits_{0}^{t}v_{p}\left( s\right) \,ds\right) ^{\frac{p\left( k-2\right) }{%
k\left( p-2\right) }}\leqslant \\
\leqslant K_{2}\left( p\right) \mathbf{M}\sum\limits_{k=2}^{p}\left[ \left(
\int\limits_{0}^{t}v_{2}^{p/2}\left( s\right) \,ds\right) ^{\frac{1}{k\left(
p-2\right) }}+\left( \int\limits_{0}^{t}v_{p}\left( s\right) \,ds\right) ^{%
\frac{1}{k\left( p-2\right) }}\right] ^{\left( p-2\right) k}\leqslant \\
\leqslant K_{2}\left( p\right) \sum\limits_{k=2}^{p}2^{\left( p-2\right) k-1}%
\mathbf{M}\int\limits_{0}^{t}\left[ v_{2}^{p/2}\left( s\right) +v_{p}\left(
s\right) \right] \,ds\leqslant K_{p}^{p}\left\Vert \alpha\right\Vert
_{L_{p}\left( \Pi\right) }^{p}.
\end{gather*}
Hence, 
\begin{equation*}
\left\vert \mathcal{P}\left( \alpha\right) \left( t\right) \right\vert
_{p}\leqslant K_{p}\left\Vert \alpha\right\Vert _{L_{p}\left( \Pi\right) }
\end{equation*}
where $K_{1}\left( p\right) ,K_{2}\left( p\right) ,K_{p}$ are functions of $p
$. Because $\mathcal{P}\left( \alpha\right) \left( t\right) $ is martingale,
then according to (\ref{V80f2}) 
\begin{equation*}
\left\vert \mathcal{P}\left( \alpha\right) \right\vert _{p}^{\left( d\right)
}\leqslant p^{\prime}\left\vert \mathcal{P}\left( \alpha\right) \right\vert
_{p}\leqslant p^{\prime}K_{p}\left\Vert \alpha\right\Vert _{L_{p}\left(
\Pi\right) }, 
\end{equation*}
that completes the proof
\end{proof}

\subsection{Operator $\tilde{P}$}

To study the structure of operator $\mathcal{P}^{\ast}$ let define in
analogous to $N_{2}\left( \mathcal{B}\times\mathcal{BF}\right) $ the Hilbert
space $N_{2}\left( \mathcal{B}\times\mathcal{BF}\times\mathcal{B}^{l}\right) 
$ of $\mathcal{B}\times\mathcal{BF}\times\mathcal{B}^{l}$-measurable
functions $\mu\left( t,s,\omega,y\right) $, $t\in\left[ 0,1\right] $, $s\in%
\left[ 0,t\right] $, $\omega\in\Omega$, $y\in R^{l}$ (classes of equivalent
functions with respect to measure $\limfunc{mes}\times \limfunc{mes}\times%
\mathbf{P}\times\Pi$) with the norm%
\begin{equation*}
\left\vert \left\vert \left\vert \mu\right\vert \right\vert \right\vert
_{\Pi }=\left\{ \mathbf{M}\int\limits_{0}^{1}\int\limits_{0}^{t}\int\mu^{2}%
\left( t,s,\omega,y\right) \,\Pi\left( dy\right) dsdt\right\} ^{1/2}. 
\end{equation*}
In analogous to $T_{r}$ and $\mathcal{\tilde{J}}$ define the lineal $%
T_{2}\left( \nu\right) \subset N_{2}\left( \mathcal{B}\times \mathcal{BF}%
\times\mathcal{B}^{l}\right) $ and the operator $\mathcal{\tilde {P}}$ on
it; then extend this operator to the closer $\overline{T_{2}\left(
\nu\right) }$ by the continuity. Instead (\ref{V80f12}) we use now the
equality%
\begin{equation*}
\left\Vert \mathcal{\tilde{P}}\left( \mu\right) \right\Vert _{2}=\left\vert
\left\vert \left\vert \mu\right\vert \right\vert \right\vert _{\Pi},\;\mu\in
T_{2}\left( \nu\right) . 
\end{equation*}

Proofs of the next sentences are similar to the proofs of lemmas \ref{Lr}
and \ref{ad}.

\begin{lemma}
\label{adP}\textsf{\ }\textit{Let }$\mu\in\overline{T_{2}\left( \nu\right) }$%
\textit{, }$\eta\left( t\right) =\mathcal{\tilde{P}}\left( \psi\right)
\left( t\right) $\textit{. Then}
\end{lemma}

\begin{description}
\item[a)] $\mathbf{M}\eta\left( t\right) =0$ for almost all\textit{\ }$t\in%
\left[ 0,1\right] ;$

\item[b)] $\forall t\in\left[ 0,1\right] $ and almost all\textit{\ }$s\in%
\left[ t,1\right] $ \textit{a.s.} 
\begin{equation*}
\mathbf{M}\left\{ \left. \eta\left( s\right) \right\vert \mathcal{F}%
_{t}\right\} =\int\limits_{0}^{t}\int\mu\left( s,\tau,\omega,y\right) \,%
\tilde{\nu}\left( d\tau,dy\right) ; 
\end{equation*}

\item[c)] if\textit{\ }$\alpha\in L_{2}\left( \Pi\right) $\textit{, then for
almost all }$t$%
\begin{equation*}
\mathbf{M}\eta\left( t\right) \mathcal{P}\left( \alpha\right) \left(
t\right) =\mathbf{M}\int\limits_{0}^{t}\int\mu\left( t,\tau,\omega,y\right)
\alpha\left( \tau,\omega,y\right) \Pi\left( dy\right) d\tau; 
\end{equation*}

\item[d)] $\forall\tau\in\left( 0,1\right] $ \textit{a.s.} 
\begin{equation*}
\int\limits_{0}^{\tau}\eta\left( s\right)
\,ds=\int\limits_{0}^{\tau}\int\left( \int\limits_{t}^{\tau}\mu\left(
s,t,\omega,y\right) \,ds\right) \tilde{\nu}\left( dt,dy\right) . 
\end{equation*}
\end{description}

\begin{lemma}
\label{L2P}The set\textit{\ }$L_{2}\left( \nu\right) \equiv\left\{ \mathcal{%
\tilde{P}}\left( \mu\right) ,\mu\in\overline{T_{2}\left( \nu\right) }%
\right\} $ is a closed subset of\textit{\ }$L_{2}$\textit{.}
\end{lemma}

\subsection{Conjugate operator $\mathcal{P}^{\ast}$}

Let $L_{2}^{\perp}\left( \nu\right) $ is the orthogonal complement of $%
L_{2}\left( \nu\right) $ to Hilbert space $L_{2}=L_{2}\left( \nu\right)
\oplus L_{2}^{\perp}\left( \nu\right) $. Using lemmas \ref{adP},\ref{L2P},
we can prove the following theorem.

\begin{theorem}
The restriction of operator\textit{\ }$\mathcal{P}^{\ast}:\left(
DH_{p}\right) ^{\ast}\rightarrow\left( L_{p}\left( \Pi\right) \right)
^{\ast},\,p\geqslant2$ \textit{( }$p$ \textit{is an even number) to the
space }$L_{2}\subset\left( DH_{p}\right) ^{\ast}$ is bounded from\textit{\ }$%
L_{2}$ \textit{to }$L_{2}\left( \Pi\right) $; it can be represented in the
form%
\begin{equation*}
\mathcal{P}^{\ast}\left( \chi\right) \left( t,\omega,y\right)
=\int\limits_{t}^{1}\mu\left( s,t,\omega,y\right) \,ds, 
\end{equation*}
\textit{\ where }$\chi=\chi_{1}+\mathcal{\tilde{P}}\left( \mu\right) $%
\textit{, }$\chi_{1}\in L_{2}^{\perp}\left( \nu\right) $\textit{, }$\mu \in%
\overline{T_{2}\left( \nu\right) }$\textit{.}
\end{theorem}

\section{Some relations between operators\ $\mathcal{L}$ , $\mathcal{J}$ and 
$\mathcal{P}$}

Hilbert spaces $\mathcal{M}$ and $L_{2}$ can be expand to the direct sums of
mutually orthogonal closed subspaces:%
\begin{gather*}
\mathcal{M}=\mathcal{M}_{2}^{w}\oplus\mathcal{M}_{2}^{\nu}\oplus \mathcal{M}%
_{2}^{\perp}, \\
L_{2}=L_{2}\left( w\right) \oplus L_{2}^{\perp}\left( w\right) =L_{2}\left(
\nu\right) \oplus L_{2}^{\perp}\left( \nu\right) =L_{2}\left( w\right)
\oplus L_{2}\left( \nu\right) \oplus L_{2}^{\perp}.
\end{gather*}

\begin{lemma}
\textit{Let} $\chi\in L_{2}$\textit{, }$\theta\equiv\mathcal{L}^{\ast}\left(
\chi\right) $\textit{, }$\varkappa\equiv\mathcal{J}^{\ast}\left( \chi\right) 
$\textit{, }$\alpha\equiv\mathcal{P}^{\ast}\left( \chi\right) $\textit{.
Then there exist a martingale }$h\in\mathcal{M}_{2}^{\perp}$\textit{\ such
that}%
\begin{equation*}
\theta\left( t\right) =-\int\limits_{0}^{t}\chi\left( s\right)
ds+\int\limits_{0}^{t}\varkappa\left( s\right) dw\left( s\right)
+\int\limits_{0}^{t}\int\alpha\left( s,y\right) d\tilde{\nu}\left(
ds,dy\right) +h\left( t\right) ,\;t\in\left[ 0,1\right] . 
\end{equation*}
\end{lemma}

\begin{proof}
Using formula (\ref{V80f10}) we expand martingale $\mu\in\mathcal{M}_{2}$ to
the sum%
\begin{equation*}
\mu\left( t\right) =\mathcal{J}\left( \varkappa_{1}\right) \left( t\right) +%
\mathcal{P}\left( \alpha_{1}\right) \left( t\right) +h\left( t\right) , 
\end{equation*}
where $\varkappa_{1}\in L_{2},\;\alpha_{1}\in L_{2}\left( \Pi\right) $, $h\in%
\mathcal{M}_{2}^{\perp}$. Let show that $\varkappa_{1}=\varkappa$, $%
\alpha_{1}=\alpha$. Using Fubini theorem we get for any $\varphi\in L_{2}$%
\begin{align*}
\left\langle \varkappa_{1},\varphi\right\rangle & =\mathbf{M}\int
\limits_{0}^{1}\varkappa_{1}\varphi\,ds=\mathbf{M}\int\limits_{0}^{1}%
\varkappa_{1}dw\int\limits_{0}^{1}\varphi\,dw=\mathbf{M}\mu\left( 1\right)
\int\limits_{0}^{1}\varphi\,dw= \\
& =\mathbf{M}\int\limits_{0}^{1}\chi ds\int\limits_{0}^{1}\varphi
\,dw=\int\limits_{0}^{1}\mathbf{MM}\left\{ \left. \chi\left( s\right)
\int\limits_{0}^{1}\varphi\,dw\right\vert \mathcal{F}_{s}\right\} ds= \\
& =\mathbf{M}\int\limits_{0}^{1}\chi\left( s\right)
\int\limits_{0}^{s}\varphi\,dwds=\left\langle \chi,\mathcal{J}\left(
\varphi\right) \right\rangle .
\end{align*}
This implies $\varkappa_{1}=\mathcal{J}^{\ast}\left( \chi\right) =\varkappa $%
. The equality $\alpha_{1}=\alpha$ can be established analogously
\end{proof}

\begin{lemma}
\textit{Let} $\mathcal{F}_{t}\equiv\mathcal{F}_{t}^{w}$\textit{, }$\chi\in
L_{r}$,$\,r\in\left( 1,2\right] $\textit{, }$\theta=\mathcal{L}^{\ast
}\left( \chi\right) $\textit{, }$\varkappa=\mathcal{J}^{\ast}\left(
\chi\right) $\textit{. Then for every }$t\in\left[ 0,1\right] $%
\begin{equation*}
\theta\left( t\right) =-\int\limits_{0}^{t}\chi\left( s\right)
\,ds+\int\limits_{0}^{t}\varkappa\left( s\right) \,dw\left( s\right) . 
\end{equation*}
\end{lemma}

\begin{proof}
According to Theorem \ref{T1} the martingale $\mu\left( t\right) $ can be
represented in the form%
\begin{equation*}
\mu\left( t\right) =\mathbf{M}\mu\left( t\right)
+\int\limits_{0}^{t}\varkappa_{1}\left( s\right) \,dw\left( s\right) , 
\end{equation*}
where $\varkappa_{1}\in N_{r}.$ The equality $\varkappa_{1}=\varkappa$ can
be established similar to preceding lemma (using also Lemma \ref{L2})
\end{proof}

\begin{lemma}
\textit{Let }$\chi\in L_{2}$\textit{, }$\theta=\mathcal{L}^{\ast}\left(
\chi\right) $\textit{. Expand }$\theta\in L_{2}$ to the sum%
\begin{equation*}
\theta=\mathcal{\tilde{J}}\left( \lambda\right) +\mathcal{\tilde{P}}\left(
\mu\right) +\theta_{1}, 
\end{equation*}
where\textit{\ }$\lambda\in\overline{T_{2}}$\textit{, }$\mu\in\overline{T_{2}%
}\left( \nu\right) $\textit{, }$\theta_{1}\in L_{2}^{\perp}$\textit{. Then
functions }$\lambda\left( t,\tau,\omega\right) ,\mu\left( t,\tau
,\omega,y\right) $ are continuous for\textit{\ }$t\geqslant\tau$ and a.s.%
\begin{align*}
\mathcal{J}^{\ast}\left( \chi\right) \left( t,\omega\right) & =\lambda\left(
t,t,\omega\right) , \\
\mathcal{P}^{\ast}\left( \chi\right) \left( t,\omega,y\right) & =\mu\left(
t,t,\omega,y\right) .
\end{align*}
\end{lemma}

\begin{proof}
Note that $\overline{T_{2}}=N_{2}\left( \mathcal{B}\times\mathcal{BF}\right) 
$. According to theorem 4 we have $\chi=\chi_{1}+\mathcal{\tilde{J}}\left(
\lambda_{1}\right) $, where $\lambda_{1}\in\overline{T_{2}}$, $\chi_{1}\in
L_{2}^{\perp}\left( w\right) $. Thus, $\theta=\theta_{2}+\mathcal{L}^{\ast }%
\mathcal{\tilde{J}}\left( \lambda_{1}\right) $, $\theta_{2}\equiv \mathcal{L}%
^{\ast}\left( \chi_{1}\right) $. According to (\ref{V80f10}) 
\begin{equation*}
\mathcal{L}^{\ast}\mathcal{\tilde{J}}\left( \lambda_{1}\right) \left(
t\right) =\mathbf{M}\left\{ \left. \int\limits_{0}^{1}\mathcal{\tilde{J}}%
\left( \lambda_{1}\right) \left( s\right) \,ds\right\vert \mathcal{F}%
_{t}\right\} -\int\limits_{0}^{t}\mathcal{\tilde{J}}\left( \lambda
_{1}\right) \left( s\right) \,ds. 
\end{equation*}
Transform the right-hand side of this equality 
\begin{equation*}
\begin{array}{c}
\mathbf{M}\left\{ \left. \int\limits_{0}^{1}\left(
\int\limits_{\tau}^{1}\lambda_{1}\left( s,\tau\right) \,ds\right) dw\left(
\tau\right) \right\vert \mathcal{F}_{t}\right\} -\int\limits_{0}^{t}\left(
\int\limits_{\tau}^{t}\lambda_{1}\left( s,\tau\right) \,ds\right) \,dw\left(
\tau\right) = \\ 
=\int\limits_{0}^{t}\left( \int\limits_{\tau}^{1}\lambda_{1}\left(
s,\tau\right) \,ds\right) dw\left( \tau\right) -\int\limits_{0}^{t}\left(
\int\limits_{\tau}^{t}\lambda_{1}\left( s,\tau\right) \,ds\right) \,dw\left(
\tau\right) = \\ 
=\int\limits_{0}^{t}\left( \int\limits_{t}^{1}\lambda_{1}\left(
s,\tau\right) \,ds\right) dw\left( \tau\right) ,\;t\in\left[ 0,1\right]%
\end{array}
\end{equation*}
(see item (d) of Lemma \ref{ad}). It is easy to prove that function%
\begin{equation*}
\lambda_{2}\left( t,\tau,\omega\right) \equiv\int\limits_{t}^{1}\lambda
_{1}\left( s,\tau,\omega\right) \,ds\in N_{2}\left( \mathcal{B}\times%
\mathcal{BF}\right) =\overline{T_{2}}. 
\end{equation*}
Thus, as it was mentioned in the proof of lemma 6, 
\begin{equation*}
\mathcal{L}^{\ast}\mathcal{\tilde{J}}\left( \lambda_{1}\right) \left(
t\right) =\int\limits_{0}^{t}\lambda_{2}\left( t,\tau,\omega\right)
\,dw\left( \tau\right) =\mathcal{\tilde{J}}\left( \lambda_{2}\right) \left(
t\right) 
\end{equation*}
for almost all $t$ a.s. Let us next show that $\theta_{2}\in L_{2}^{\perp
}\left( w\right) $. Indeed, if $\varphi=\mathcal{\tilde{J}}\left(
\psi\right) \in L_{2}\left( w\right) $ then for almost all $t$ a.s. 
\begin{gather*}
\mathcal{L}\left( \varphi\right) \left( t\right) =\int\limits_{0}^{t}\left(
\int\limits_{\tau}^{t}\psi\left( s,\tau,\omega\right) \,ds\right) dw\left(
\tau\right) =\mathcal{\tilde{J}}\left( \mu _{1}\right) \left( t\right) \in
L_{2}\left( w\right) , \\
\mu_{1}\left( t,\tau,\omega\right) \equiv\int\limits_{\tau}^{t}\psi\left(
s,\tau,\omega\right) \,ds\in N_{2}\left( \mathcal{B}\times\mathcal{BF}%
\right) .
\end{gather*}
Since $\chi_{1}\in L_{2}^{\perp}\left( w\right) $, it follows%
\begin{equation*}
\left\langle \varphi,\theta_{2}\right\rangle =\left\langle \varphi ,\mathcal{%
L}^{\ast}\left( \chi_{1}\right) \right\rangle =\left\langle \mathcal{L}%
\left( \varphi\right) ,\chi_{1}\right\rangle =0 
\end{equation*}
that prove our statement. Thus, we prove that%
\begin{equation*}
\theta=\mathcal{\tilde{J}}\left( \lambda_{2}\right) +\theta_{2},\;\theta
_{2}\in L_{2}^{\perp}\left( w\right) . 
\end{equation*}
Comparing this formula with the initial we receive%
\begin{equation*}
\lambda\left( t,\tau,\omega\right) =\lambda_{2}\left( t,\tau,\omega\right)
=\int\limits_{t}^{1}\lambda_{1}\left( s,\tau,\omega\right) \,ds. 
\end{equation*}
But in the other hand according to the theorem 2%
\begin{equation*}
\mathcal{J}^{\ast}\left( \chi\right) \left( t,\omega\right) =\int
\limits_{t}^{1}\lambda_{1}\left( s,t,\omega\right) \,ds. 
\end{equation*}
Combining these two equalities, we obtain the first statement of the lemma.
The second statement can be obtained analogously
\end{proof}

\end{document}